\documentclass[runningheads]{llncs}
\usepackage[T1]{fontenc}
\usepackage{bm}
\usepackage{amssymb}
\usepackage{amsmath}

\usepackage{graphicx}
\usepackage{hyperref}
\usepackage{color}

\usepackage[capitalise]{cleveref}

\usepackage{scalerel}
\usepackage{tikz}
\usetikzlibrary{svg.path}

\definecolor{orcidlogocol}{HTML}{A6CE39}
\tikzset{
  orcidlogo/.pic={
    \fill[orcidlogocol] svg{M256,128c0,70.7-57.3,128-128,128C57.3,256,0,198.7,0,128C0,57.3,57.3,0,128,0C198.7,0,256,57.3,256,128z};
    \fill[white] svg{M86.3,186.2H70.9V79.1h15.4v48.4V186.2z}
                 svg{M108.9,79.1h41.6c39.6,0,57,28.3,57,53.6c0,27.5-21.5,53.6-56.8,53.6h-41.8V79.1z M124.3,172.4h24.5c34.9,0,42.9-26.5,42.9-39.7c0-21.5-13.7-39.7-43.7-39.7h-23.7V172.4z}
                 svg{M88.7,56.8c0,5.5-4.5,10.1-10.1,10.1c-5.6,0-10.1-4.6-10.1-10.1c0-5.6,4.5-10.1,10.1-10.1C84.2,46.7,88.7,51.3,88.7,56.8z};
  }
}

\newcommand\orcidicon[1]{\href{https://orcid.org/#1}{\mbox{\scalerel*{
\begin{tikzpicture}[yscale=-1,transform shape]
\pic{orcidlogo};
\end{tikzpicture}
}{|}}}}

\begin{document}

\title{Learning discrete Lagrangians for variational PDEs from data and detection of travelling waves}
\titlerunning{Learning variational PDEs from data}
%
\author{
Christian Offen ${\protect \orcidicon{0000-0002-5940-8057}}$ 
\and
Sina Ober-Blöbaum
}
\authorrunning{Offen and Ober-Blöbaum}
%
\institute{Paderborn University, Warburger Str. 100, 33098 Germany
\email{christian.offen@uni-paderborn.de}\\
\url{https://www.uni-paderborn.de/en/person/85279}
}
\maketitle              
%

\def\d{\mathrm{d}}
\def\D{\mathrm{D}}
\def\p{\partial}
\def\R{\mathbb{R}}
\def\N{\mathbb{N}}
\def\Z{\mathbb{Z}}

\begin{abstract}
The article shows how to learn models of dynamical systems from data which are governed by an unknown variational PDE. 
Rather than employing reduction techniques, we learn a discrete field theory governed by a discrete Lagrangian density $L_d$ that is modelled as a neural network. Careful regularisation of the loss function for training $L_d$ is necessary to obtain a field theory that is suitable for numerical computations: we derive a regularisation term which optimises the solvability of the discrete Euler--Lagrange equations.
Secondly, we develop a method to find solutions to machine learned discrete field theories which constitute travelling waves of the underlying continuous PDE.

\end{abstract}

\paragraph*{Publication information}
For the published version of the article refer to \cite{DLNNDensity}.

\section{Introduction}

In data-driven system identification, a model is fitted to observational data of a dynamical system. The quality of the learned model can greatly be improved when prior geometric knowledge about the dynamical system is taken into account such as conservation laws \cite{symplecticShadowIntegrators,LagrangianShadowIntegrators,LNN,HNN,Bertalan2019}, symmetries \cite{SymLNN,SymHNN,Dierkes2021}, equilibrium points \cite{Ridderbusch2021}, or asymptotic behaviour of its motions. 

One of the most fundamental principles in physics is the variational principle: it says that motions constitute stationary points of an action functional. The presence of variational structure is related to many qualitative features of the dynamics such as the validity of Noether's theorem: symmetries of the action functional are in correspondence with conservation laws.
To guarantee that these fundamental laws of physics hold true for learned models, Greydanus et al propose to learn the action functional from observational data \cite{LNN} (Lagrangian neural network) and base prediction on numerical integrations of Euler--Lagrange equations. Quin proposes to learn a discrete action instead \cite{Qin2020}. An ansatz of a discrete model has the advantage that it can be trained with position data of motions only. In contrast, learning a continuous theory typically requires information about higher derivatives (corresponding to velocity, acceleration, momenta, for instance) which are typically not observed but only approximated. Moreover, the discrete action functional (once it is learned) can naturally be used to compute motions numerically.

However, in \cite{LagrangianShadowIntegrators} the authors demonstrate that care needs to be taken when learned action functionals are used to compute motions: even if the data-driven action functional is perfectly consistent with the training data (i.e.\ the machine learning part of the job is successfully completed), minimal errors in the initialisation process of numerical computations get amplified. 
As a remedy, the authors develop Lagrangian Shadow Integrators which mitigate these amplified numerical errors based on a technique called backward error analysis. Moreover, using backward error analysis they relate the discrete quantities to their continuous analogues and show how to analyse qualitative aspects of the machine learned model.
Action functionals are not uniquely determined by the motions of a dynamical system.
Therefore, regularisation is needed to avoid learning degenerate theories. While in \cite{LagrangianShadowIntegrators} the authors develop a regularisation strategy when the action functional is modelled as a Gaussian Process, Lishkova et al develop a corresponding regularisation technique for artificial neural networks in \cite{SymLNN} in the context of ordinary differential equations (ODEs).

In this article we show how to learn a discrete action functional from discrete data which governs solutions to partial differential equations (PDEs) using artificial neural networks extending the regularisation strategy which we have developed in \cite{SymLNN}.
Our technique to learn (discrete) densities of action functionals can be contrasted to approaches where a spatial discretisation of the problem is considered first, followed by structure-preserving model reduction techniques (data-driven or analytical) \cite{Glas2023,Carlsberg2015}
and then a model for the reduced system of ODEs is learned from data \cite{Sharma2022LagrangianROM,Blanchette2022,Blanchette2020}.

Travelling waves solutions of PDEs are of special interest due to their simple structure. When a discrete field theory for a continuous process described by a PDE is learned, they typically "get lost" because the mesh of the discrete theory is incompatible with certain wave speeds. In this article, we introduce a technique to find the solutions of data-driven discrete theories that correspond to travelling waves in the underlying continuous dynamics (shadow travelling waves).
The article contains the following novelties:
\begin{itemize}
\item We transfer our Lagrangian ODE regularising strategy \cite{SymLNN} to data-driven discrete field theories in a PDE setting and provide a justification using numerical analysis. 

\item The development of a technique to detect travelling waves in data-driven discrete field theories.

\end{itemize}

The article proceeds with a review of variational principles (\cref{sec:review}), an introduction of our machine learning architecture and derivation of the regularisation strategy (\cref{sec:mlarchitectures}). In \cref{sec:TW}, we define the notion of {\em shadow travelling waves} and show how to find them in data-driven models. The article concludes with numerical examples relating to the wave equation (\cref{sec:Experiment}).

\section{Discrete and continuous variational principles}\label{sec:review}

\subsubsection{Continuous variational principles}

Many differential equations describing physical phenomena such as waves, the state of a quantum system, or the evolution of a relativistic fields are derived from a variational principle: solutions are characterised as critical points of a (non-linear) functional $S$ defined on a suitable space of functions $u \colon X \to \R^d$ and has the form
\begin{equation}\label{eq:SGeneral}
S(u) = \int_X L(\bm{x},u(\bm{x}),u_{x_0}(\bm{x}),u_{x_1}(\bm{x}),\ldots,u_{x_n}(\bm{x})) \d \bm{x},
\end{equation}
where $\bm{x}= (x_0,x_1,\ldots,x_n) \in X$ and $u_{x_j}$ denote partial derivatives of $u$. This variational principle can be referred to as a {\em first-order field theory}, since only derivatives to the first order of $u$ appear. In many applications the free variable $x_0$ corresponds to time and is denoted by $x_0=t$.
The functional $S$ is stationary at $u$ with respect to all variations $\delta u \colon X \to \R$ vanishing at the boundary (or with the correct asymptotic behaviour) if and only if the Euler--Lagrange equations
\begin{equation}\label{eq:EL}
0 =\mathrm{EL}(L) 
= \frac{\p L}{\p u} - \sum\nolimits_{j=0}^n \frac{\d}{\d x_j}\frac{\p L}{\p u_{x_j}}
\end{equation}
are fulfilled on $X$.

\begin{example}\label{ex:wave}
The wave equation
\begin{equation}\label{eq:wavePDE}
u_{tt}(t,x)-u_{xx}(t,x) + \nabla V(u(t,x))=0
\end{equation}
is the Euler--Lagrange equation $0 =\mathrm{EL}(L)$ to the Lagrangian
\begin{equation}\label{eq:waveL}
L(u,u_t,u_x)= \frac{1}2 (u_t^2-u_x^2)-V(u).
\end{equation}
Here $\nabla V$ denotes the gradient of a potential $V$.
\end{example}

\begin{remark}\label{rem:LGaugeFreedom}
Lagrangians are not uniquely determined by the motions of a dynamical system: two first order Lagrangians $L$ and $\tilde L$ yield equivalent Euler--Lagrange equations if $s L-\tilde L$ ($s \in \R \setminus \{0\}$) is a total divergence $\nabla_{\bm x} \cdot F(\bm x,u(\bm x)) = \frac{\p }{\p x_1}(F^1(\bm x,u(\bm x)))+\ldots +\frac{\p }{\p x_n}F^n(\bm x,u(\bm x))$ for $F = (F^1,\ldots,F^n)\colon X\times \R^d \to \R^n$. 
\end{remark}


\subsubsection{Discrete variational principle}\label{sec:DiscreteDEL}

For simplicity, we consider the two dimensional compact case: let $X=[0,T] \times [0,l]/\{0,l\}$ with $T,l >0$. Here $[0,l]/\{0,l\}$ is the real interval $[0,l]$ with identified endpoints (periodic boundary conditions). Consider a uniform, rectangular mesh $X_\Delta$ on $X$ with mesh widths $\Delta t=\frac TN$ and $\Delta x = \frac  lM$ for $N,M \in \N$.
A discrete version of the action functional \eqref{eq:SGeneral} is 
\begin{equation*}\label{eq:SdGeneral}
S_d \colon (\R^d)^{(N-1)\times M} \to \R, \quad S_d(U)
= \Delta t \Delta x \sum_{i=1}^{N-1}\sum_{j=0}^{M-1}
L_d(u_j^i,u_j^{i+1},u_{j+1}^{i})
\end{equation*}
for a discrete Lagrangian density $L_d \colon (\R^d)^3 \to \R$ together with temporal boundary conditions for $u^{0}_j \in \R^d$ and $u^N_j \in \R^d$ for $j=0,\ldots,M-1$. Above $U$ denotes the values $(u^{i}_{j})_{j=0,\ldots N-1}^{i=1,\ldots M-1}$ on inner mesh points. We have $u^{i}_{M}=u^{i}_0$ by the periodicity in space.
Solutions of the variational principle are $U \in (\R^d)^{(N-1)\times M}$ such that $U$ is a critical point of $S_d$. This is equivalent to the condition that for all $i = 1,\ldots,N-1$ and $j=0,\ldots,M-1$ the discrete Euler--Lagrange equations
\begin{equation}\label{eq:DEL}
\begin{split}
\frac{\p}{\p u^i_j}
\big(
L_d(u^{i}_{j},u^{i+1}_{j},u^{i}_{j+1})
+ L_d(u^{i-1}_{j},u^{i}_{j},u^{i-1}_{j+1})
+ L_d(u^{i}_{j-1},u^{i+1}_{j-1},u^{i}_{j})
\big)
=0
\end{split}
\end{equation}
are fulfilled. The expression on the left of \eqref{eq:DEL} is abbreviated as $\mathrm{DEL}(L_d)^{i}_{j}(U)$ in the following.

\begin{remark}\label{rem:boundary}
Instead of periodic boundary conditions in space, $S_d$ can be adapted to other types of boundary conditions
such as Dirichlet- or Neumann conditions.
\end{remark}

\begin{example}\label{ex:waveDiscrete}
The discretised wave equation
\begin{equation}\label{eq:waveDiscretised}
\frac{(u^{i-1}_{j}-2u^{i}_{j}+u^{i+1}_{j})}{\Delta t^2}
-\frac{(u^{i}_{j-1}-2u^{i}_{j}+u^{i}_{j+1})}{\Delta x^2}
+ \nabla V(u^{i}_{j})=0
\end{equation}
is the discrete Euler--Lagrange equations to the discrete Lagrangian
\begin{equation*}\label{eq:LdWave}
L_d(u^{i}_{j},u^{i+1}_{j},u^{i}_{j+1})
= \frac 12 \left(\frac {u^{i+1}_{j}-u^{i}_{j}}{\Delta t^2}\right)^2
- \frac 12 \left(\frac {u^{i}_{j+1}-u^{i}_{j}}{\Delta x^2}\right)^2
- V(u^{i}_{j}).
\end{equation*}

\end{example}

\begin{remark}\label{rem:LdGaugeFreedom}
	In analogy to \cref{rem:LGaugeFreedom}, notice that $L_d$ and $\tilde{L}_d$ yield the same discrete Euler--Lagrange equations \eqref{eq:DEL} if
	\begin{equation}\label{eq:LdGaugeFreedom}
		L_d(a,b,c) - s\tilde{L}_d(a,b,c) = \chi_1(a)-\chi_1(b) + \chi_2(a)-\chi_2(c) + \chi_3(b)-\chi_3(c)
	\end{equation}
	for differentiable functions $\chi_1,\chi_2,\chi_3 \colon X \to \R$ and $s \in \R \setminus \{0\}$.
\end{remark}

\begin{remark}\label{rem:SolveDEL}
	If $\mathrm{DEL}(L_d)^{i}_{j}(U)=0$ (see \eqref{eq:DEL}) and if $\frac{\p^2 L_d}{\p u^{i}_{j}\p u^{i+1}_{j}}(u^{i}_{j},u^{i+1}_{j},u^{i}_{j+1})$ is of full rank, then \eqref{eq:DEL} is solvable for $u^{i+1}_{j}$ as a function of $u^{i}_{j}$, $u^{i}_{j+1}$, $u^{i-1}_{j}$, $u^{i-1}_{j+1}$, $u^{i}_{j-1}$, $u^{i+1}_{j-1}$ by the implicit function theorem locally around a solution of \eqref{eq:DEL}.
	All of these points correspond to mesh points that either lie to the left or below the point with indices $(i,j)$.
	If $u^1_{j}$ is known for $0 \le j\le M-1$, then utilising the boundary conditions $u^0_j \in \R^d$ and $u^{i}_{0}=u^{i}_{M}$ we can compute $U$ by subsequently solving \eqref{eq:DEL}. This corresponds to the computation of a time propagation.
\end{remark}

The following Proposition analyses the convergence of Newton-Iterations when solving \eqref{eq:DEL} for $u^{i+1}_{j}$, as is required to compute time propagations. It introduces a quantity $\rho^\ast$ that relates to how well the iterations converge. We will make use of this quantity in the design of our machine learning framework.

\begin{proposition}\label{prop:SolveDELNewton}
Let $u^{i}_{j}$, $u^{i+1}_{j}$, $u^{i}_{j+1}$, $u^{i-1}_{j}$, $u^{i-1}_{j+1}$, $u^{i}_{j-1}$, $u^{i+1}_{j-1}$ such that \eqref{eq:DEL} holds. Let $O\subset \R^d$ be a convex, neighbourhood of $u^\ast=u^{i+1}_{j}$, $\| \cdot \|$ a norm of $\R^d$ inducing an operator norm on $\R^{d \times d}$.
Define $p(u) :=\frac{\p^2 L_d}{\p u^{i}_{j}\p u}(u^i_j,u,u^i_{j+1})$ and let $\theta$ and $\overline{\theta}$ be Lipschitz constants on $O$ for $p$
and for $\mathrm{inv} \circ p$, respectively, where $\mathrm{inv}$ denotes matrix inversion.
Let
\begin{equation}\label{eq:rho}
\rho^\ast := \left\|\mathrm{inv}(p(u^\ast))\right\|
= \left\|\left(\frac{\p^2 L_d}{\p u^{i}_{j}\p u^\ast}(u^i_j,u^\ast,u^i_{j+1})\right)^{-1}\right\|
\end{equation}
and let $f(u^{(n)})$ denote the left hand side of \eqref{eq:DEL} with $u_j^{i+1}$ replaced by $u^{(n)}$.
If $\|{u}^{(0)}-{u}^\ast\| \le \min\left(\frac{\rho^\ast}{\overline{\theta}}, \frac{1}{2\theta\rho^\ast}\right)$ for ${u}^{(0)}\in \mathcal{O}$, then the Newton Iterations
${u}^{(n+1)}:= {u}^{(n)} - \mathrm{inv}(p(u^{(n)})) f(u^{(n)})$
converge quadratically against ${u}^{\ast}$, i.e.\
\begin{equation}\label{eq:NewtonQuadraticConvergence}
\|{u}^{(n+1)} - {u}^{\ast} \| \le \rho^\ast \theta \|{u}^{(n)} - {u}^{\ast} \|^2. 
\end{equation}

\end{proposition}

\begin{proof}
The statement follows from an adaption of the standard estimates for Newton's method (see \cite[\S 4]{Deuflhard2003}, for instance) to the considered setting. A detailed proof of the Proposition is contained in the Appendix (Preprint/ArXiv version only).	
\end{proof}

\begin{remark}
The assumptions formulated in \cref{prop:SolveDELNewton} are sufficient but not sharp. The main purpose of the proposition is to identify quantities that are related to the efficiency of our numerical solvers and to use this knowledge in the design of machine learning architectures.
\end{remark}

\section{Machine learning architecture for discrete field theories}\label{sec:mlarchitectures}

We model a discrete Lagrangian $L_d$ as a neural network and fit its parameters
\begin{itemize}
\item such that the discrete Euler--Lagrange equations \eqref{eq:DEL} for the learned $L_d$ are consistent with observed solutions $U=(u^{i}_{j})$ of \eqref{eq:DEL}
\item  and such that \eqref{eq:DEL} is easily solvable for $u^{i+1}_{j}$ using iterative numerical methods, so that we can use the discrete field theory to predict solutions via forward propagation of initial conditions (see \cref{rem:SolveDEL}).
\end{itemize}

For given observations $U^{(1)},\ldots,U^{(K)}$ with $U^{(k)}=({u^{i}_{j}}^{(k)})$ 
on the interior mesh $X_\Delta$, we consider the loss function $\ell = \ell_{\mathrm{DEL}} + \ell_{\mathrm{reg}}$ consisting of a data consistency term $\ell_{\mathrm{DEL}}$ and a regularising term $\ell_{\mathrm{reg}}$. We have
\begin{equation}\label{eq:lossDEL}
\ell_{\mathrm{DEL}} = \sum_{k=1}^{K} \sum_{i=1}^{N-1} \sum_{j=0}^{M-1} \mathrm{DEL}(L_d)^{i}_{j}(U^{(k)})^2
\end{equation}
with $\mathrm{DEL}(L_d)^{i}_{j}$ from \eqref{eq:DEL}. $\ell_{\mathrm{DEL}}$ measures how well $L_d$ fits to the training data.

Since a discrete Lagrangian $L_d$ is not uniquely determined by the system's motions by \Cref{rem:LdGaugeFreedom} (indeed, $L_d \equiv \mathrm{const}$ is consistent with any observed dynamics), careful regularisation is required.
Indeed, in \cite{LagrangianShadowIntegrators} we demonstrate in an ode setting that if care is not taken, then machine learned models for $L_d$ can be unsuitable for numerical purposes and amplify errors of numerical integration schemes.
In view of \cref{prop:SolveDELNewton}, we aim to minimize $\rho^\ast$ (see \eqref{eq:rho}) and define the regularisation term
\begin{equation}\label{eq:lossReg}
\ell_{\mathrm{reg}} 
= \sum_{k=1}^{K} \sum_{i=1}^{N-1} \sum_{j=0}^{M-1}
\left\|\left(\frac{\p^2}{\p u^{i}_{j}\p u^{i+1}_{j}}
L_d\left({u^{i}_{j}}^{(k)},{u^{i+1}_{j}}^{(k)},{u^{i}_{j+1}}^{(k)}\right) \right)^{-1}\right\|^{2}.
\end{equation}
In our experiments, we use the spectral norm in \eqref{eq:lossReg}, which is the operator norm induced by the standard Eucledian vector norm on $\R^d$.
Let $A_\mathrm{reg}^{i,j,k} := \frac{\p^2}{\p u^{i}_{j}\p u^{i+1}_{j}}
L_d\left({u^{i}_{j}}^{(k)},{u^{i+1}_{j}}^{(k)},{u^{i}_{j+1}}^{(k)}\right)$, i.e.\ the summands of $\ell_{\mathrm{reg}} $ are $\|(A_\mathrm{reg}^{i,j,k})^{-1}\|^2 = \frac 1 {\lambda_{\min}^2}$, where $\lambda_{\min}$ is the singular value $\lambda_{\min}$ of $A_\mathrm{reg}^{i,j,k}$ with the smallest absolute norm.
If $u^{i}_{j} \in \R^d$ with $d=1$, then $\|(A_\mathrm{reg}^{i,j,k})^{-1}\|$ can be evaluated without problems.
Otherwise, $\lambda_{\min}^2$ is computed as the smallest eigenvalue of the symmetric matrix $(A_\mathrm{reg}^{i,j,k})^\top A_\mathrm{reg}^{i,j,k}$. The eigenvalue can be approximated by inverse matrix vector iterations \cite[\S 5]{Deuflhard2003} or computed exactly if the dimension $d$ is small.

\section{Periodic travelling waves}\label{sec:TW}

For simplicity, we continue within the two-dimensional space time domain $X=[0,T] \times [0,l]/\{0,l\}$ with periodic boundary conditions in space introduced in \cref{sec:DiscreteDEL}. A periodic travelling wave (TW) of a pde on $X$ is a solution of the form $u(t,x)=f(x-ct)$ for $c \in \R$ and with $f \colon [0,l]/\{0,l\} \to \R^d$ defined on the periodic spatial domain. 
Due to their simple structure, TWs are important solutions to pdes. While the defining feature of a TW is its symmetry $u(t+s,x+s c)=u(t,x)$ for $s \in \R$, evaluated on a mesh $X_\Delta$, no such structure is evident unless the quotient $c\Delta t / \Delta x$ is rational and $T$ sufficiently large. However, after a discrete field theory is learned defined by its discrete Lagrangian $L_d$, it is of interest, whether the underlying continuous PDE has TWs. As in \cite{PDEBEA} we define {\em shadow travelling waves} (TWs) of \eqref{eq:DEL} as solutions to the functional equation
\begin{equation}\label{eq:TWFunctional}
\begin{split}
0=&\p_1 L_d(f(\xi),f(\xi-c\Delta t),f(\xi+\Delta x))\\
+&\p_2 L_d(f(\xi+c\Delta t),f(\xi),f(\xi+c\Delta t+\Delta x))\\
+&\p_3 L_d(f(\xi-\Delta x),f(\xi-c\Delta t-\Delta x),f(\xi))
\end{split}
\end{equation}
where $\p_j L_d$ denotes the partial derivative of $L_d$ with respect to its $j$th slot.


\begin{example}\label{ex:twDiscreteWave}
A Fourier series ansatz for $f$ reveals that the discrete wave equation \eqref{eq:waveDiscretised} with potential
$V(u)=\frac 12 u^2$
away from resonant cases TWs are $u(t,x)=f(x-c_n t)$ with
\begin{equation}\label{eq:twDiscreteWave}
f(\xi)=\alpha \sin(\kappa_n \xi) + \beta \cos(\kappa_n \xi),\;
\kappa_n = \frac{2 \pi n}{l}, \; n \in \Z, \alpha,\beta \in \R
\end{equation}
and with wave speed $c_n$ a real solution of
\begin{equation}\label{eq:twDiscreteWaveC}
\cos(\kappa_n c_n \Delta t) = 1 - \frac{\Delta t^2}{2}+ \frac{\Delta t^2}{\Delta x^2}(\cos(\kappa_n \Delta x))-1).
\end{equation}
A contour plot for $n=1$ is shown to the left of \cref{fig:TravellingWaveExperiment}.

\end{example}

\begin{remark}\label{rem:Palais}
The TW equation \eqref{eq:TWFunctional} inherits variational structure from the underlying PDE:
an application of Palais' principle of criticality \cite{palais1979}
of the action of $(\R,+)$ on the Sobolev space $H^1(X_c,\R)$ with $X_c=[0,l/c]\times [0,l]$ defined by $(s.u)(t,x):=u(t+s,x+cs)$ to the functional $S(u)=\int_{X_c}L_d(u(t,x),u(t+\Delta t,x),u(t,x+\Delta x)) \d t$ reveals that \eqref{eq:TWFunctional} is governed by a formal 1st order variational principle.
This is investigated more closely in \cite{PDEBEA}.
\end{remark}

To identify TWs in a machine learned model of a discrete field theory, we make an ansatz of a discrete Fourier series $f(\xi) = \sum_{m=-\frac{M-1}{2}}^{-\frac{M}{2}} \hat f_{|m|} \exp(m \frac{2\pi i}{l} \xi)$, where bounds of the sum are rounded such that we have $M$ summands. To locate a TW, the loss function $\ell_{\mathrm{TW}}(c,\bm{\hat{f}})+\ell_{\mathrm{TW}}^{\mathrm{reg}}(c,\bm{\hat{f}})$ is minimised with
\begin{equation}\label{eq:lossTW}
\ell_{\mathrm{TW}}(c,\bm{\hat{f}})
= \sum_{i=1}^{N-1}\sum_{j=0}^{M-1} \| \mathrm{DEL}^{i}_{j}(U) \|^2,
\qquad
U=\big(f(i\Delta t - c j \Delta x) \big)_{0\le i \le N}^{0 \le j \le M-1}
\end{equation}
and regularisation $\ell_{\mathrm{TW}}^{\mathrm{reg}}=\exp(-100 \|U\|_{l^2}^2)$ with discrete $l^2$-norm $\| \cdot \|_{l^2}$ to avoid trivial solutions. Here $\bm {\hat{f}} = (\hat{f}_m)_m$.

\section{Experiment}\label{sec:Experiment}

\subsubsection{Creation of training data}
We use the space-time domain $X$ (\cref{sec:DiscreteDEL}) with $T=0.5$, $l=1$, $\Delta x=0.05$, $\Delta t=0.025$.
To obtain training data that behaves like discretised smooth functions, we compute $K=80$ solutions to the discrete wave equation (\cref{ex:waveDiscrete}) with potential $V(u)=\frac 12 u^2$ on the mesh $X_\Delta$ from initial data ${\bm u^0} = (u^{0}_{j})_{0\le j \le M-1}$ and ${\bm u^1} = (u^{1}_{j})_{0\le j \le M-1}$.
To obtain ${\bm u^0}$ we sample $r$ values from a standard normal distribution. Here $r$ is the dimension of the output of a real discrete Fourier transformation of an $M$-dimensional vector. These are weighted by the function $m \mapsto M \exp(-2 j^4)$, where $m=0,\ldots,r-1$ is the frequency number.
The vector ${\bm u^0}$ is then obtained as the inverse real discrete Fourier transform of the weighted frequencies.
To obtain ${\bm u^1}$ an initial velocity field ${\bm v^0} = (v^{0}_{j})_{0\le j \le M-1}$ is sampled from a standard normal distribution. Then we proceed as in a variational discretisation scheme \cite{MarsdenWestVariationalIntegrators} applied to the Lagrangian density $L$ of the continuous wave equation (\cref{ex:wave}): to compute conjugate momenta we set $L_\Sigma({\bm u},{\bm v}) = \sum_{j=0}^{M-1} \Delta x L(u_j,v_j)$ and compute ${\bm p^0} = \frac{\p L_\Sigma}{\p {\bm v^0}}({\bm u^0},{\bm v^0})$. Then ${\bm p^0} = -L_\Sigma({\bm u^1},({\bm u^1} - {\bm u^0})/\Delta t)$ is solved for ${\bm u^1}$. A plot of an element of the training data set is displayed in \cref{fig:WaveExperiment}.
(TWs are {\em not} part of the training data.)


\begin{figure}
	\centering
	\includegraphics[width=0.3\linewidth]{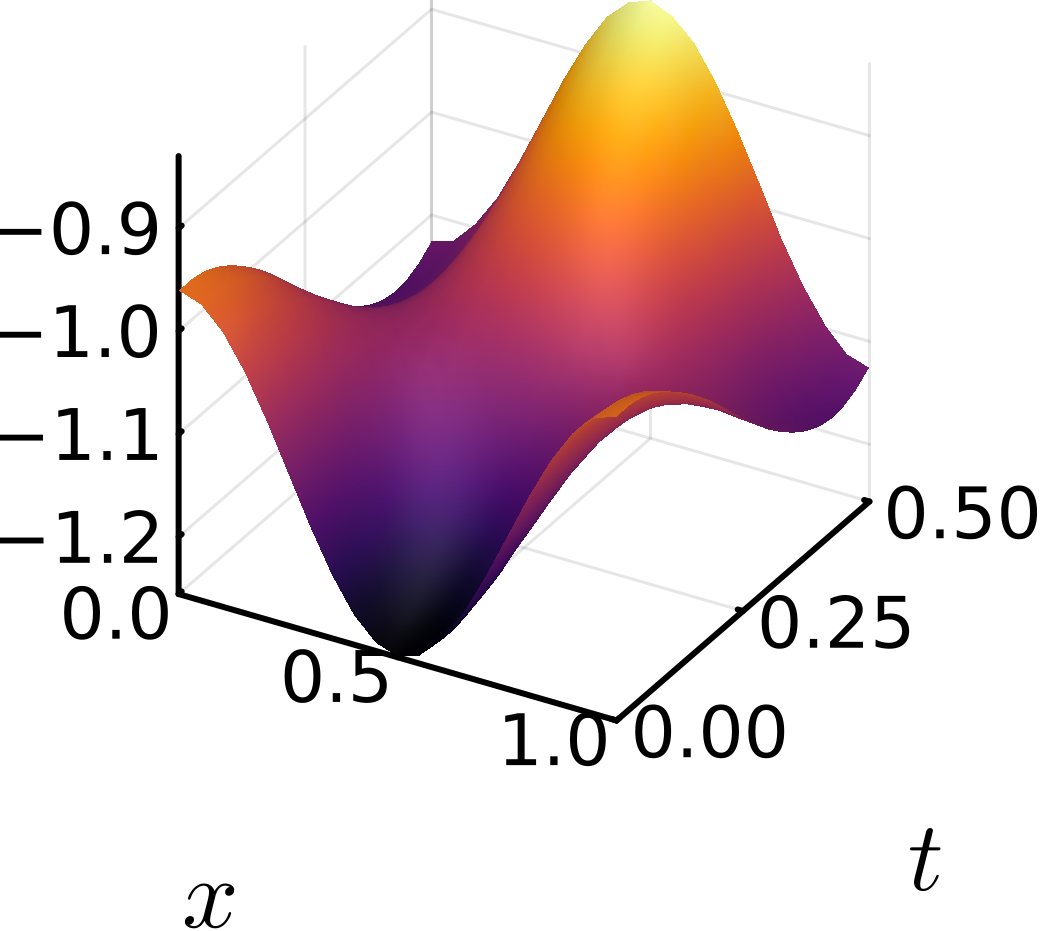}
	\includegraphics[width=0.3\linewidth]{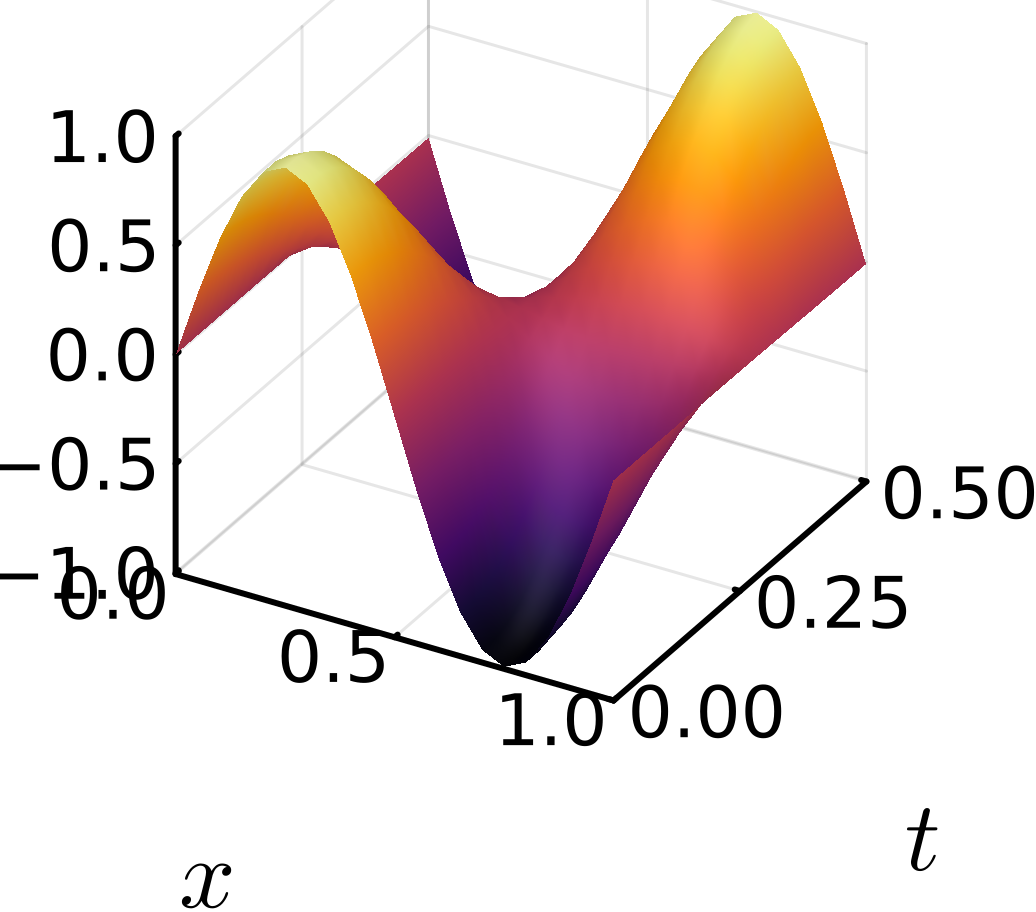}
	\includegraphics[width=0.3\linewidth]{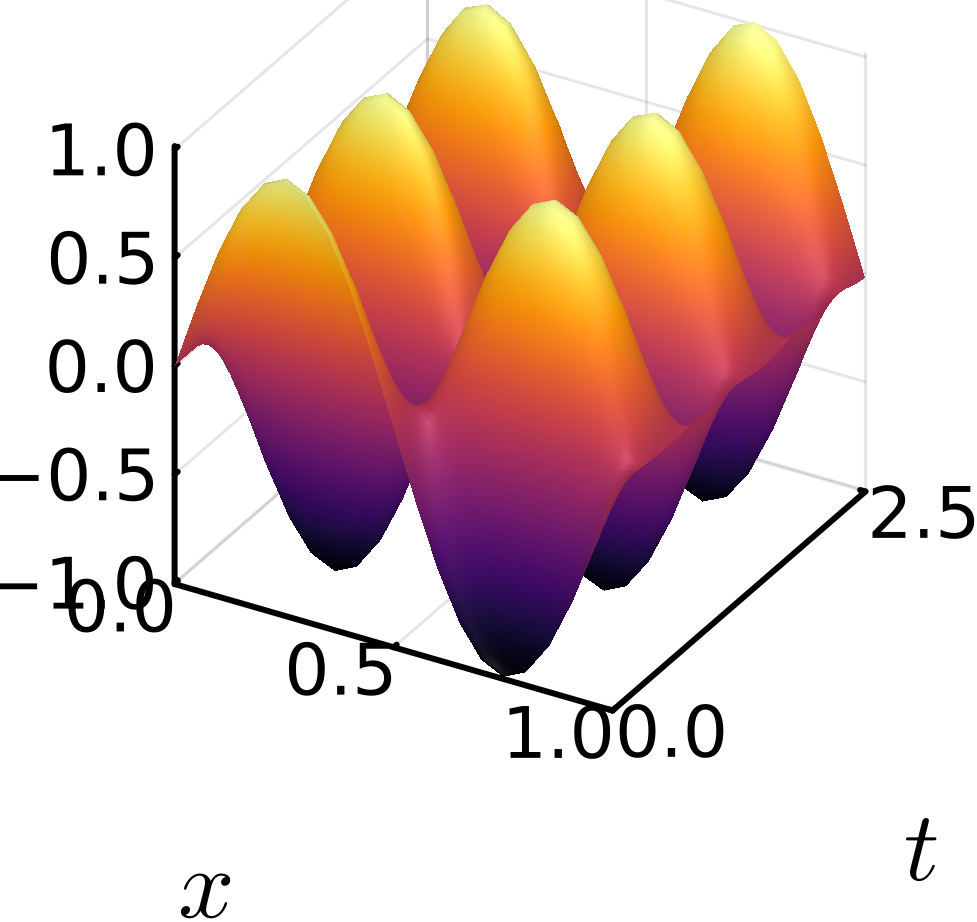}
	\caption{Left: Element of training data set. Centre: Predicted solution to unseen initial values. Right: Continued solution from centre plot outside training domain}\label{fig:WaveExperiment}
\end{figure}

\subsubsection{Training and Evaluation}
A discrete Lagrangian $L_d \colon \R \times \R \times \R \to \R$ is modelled as a three layer feed-forward neural network, where the interior layer has 10 nodes (160 parameters in total). 
It is trained on the aforementioned training data set and loss function $\ell = \ell_{\mathrm{DEL}}+\ell_{\mathrm{reg}}$ using the optimiser {\tt adam}. We perform 1320 epochs of batch training with batch size 10. For the trained model we have $\ell_{\mathrm{DEL}} \approx 8.6\cdot 10^{-8}$ and $\ell_{\mathrm{reg}} \approx 1.4 \cdot 10^{-7}$.
To evaluate the performance of the trained model for $L_d$, we compute solutions to initial data by forward propagation (\cref{rem:SolveDEL}) and compare with solutions to the discrete wave equation (\cref{ex:waveDiscrete}).
For initial data ${\bm u^0}$, ${\bm u^1}$ {\em not} seen during training, the model recovers the exact solution up to an absolute error $\|U-U_{\mathrm{ref}}\|_\infty < 0.012$ on $X_\Delta$ and up to $\|U-U_{\mathrm{ref}}\|_\infty < 0.043$ on an extended grid with $T_{\mathrm{ext}} =2.5$ (\cref{fig:WaveExperiment}).

\begin{figure}\centering
	\includegraphics[width=0.3\linewidth]{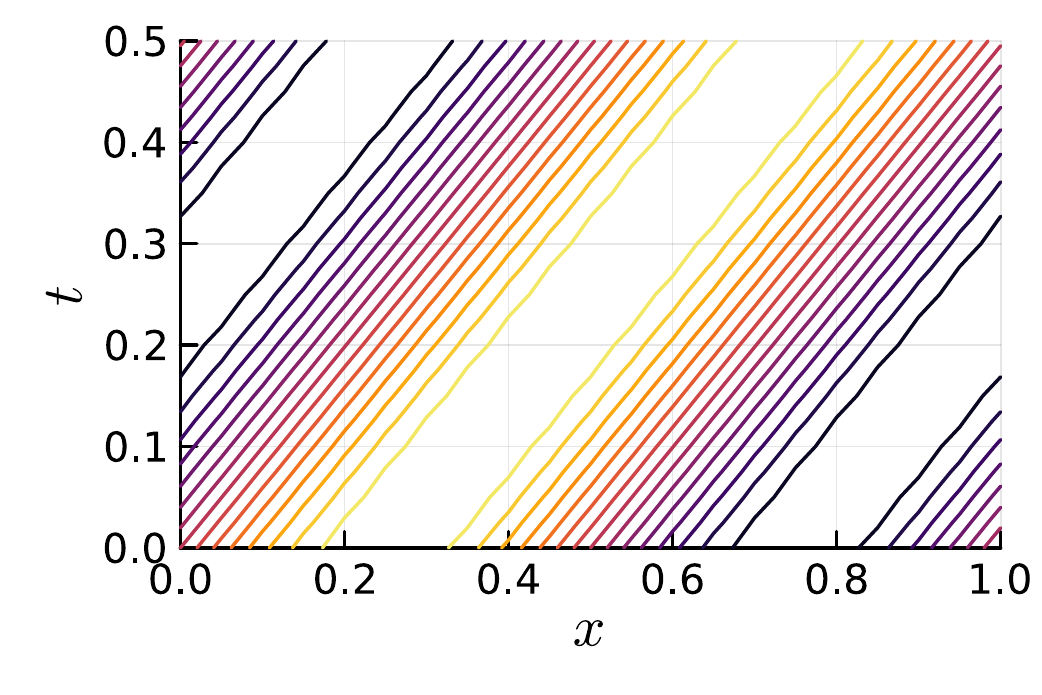}\;
	\includegraphics[width=0.3\linewidth]{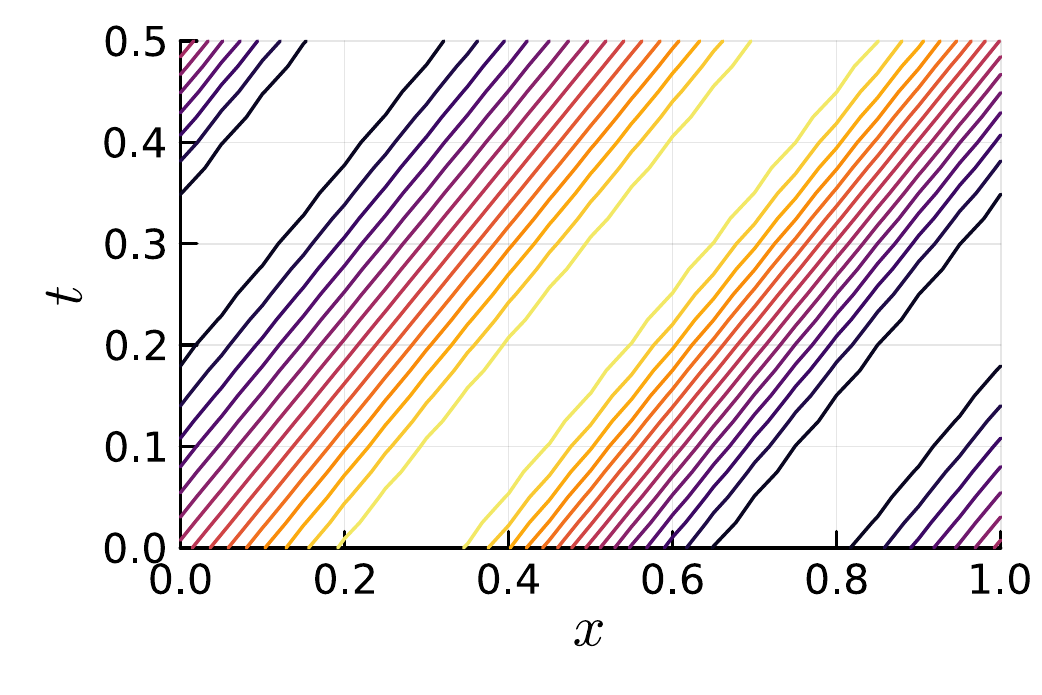}\;
	\includegraphics[width=0.27\linewidth]{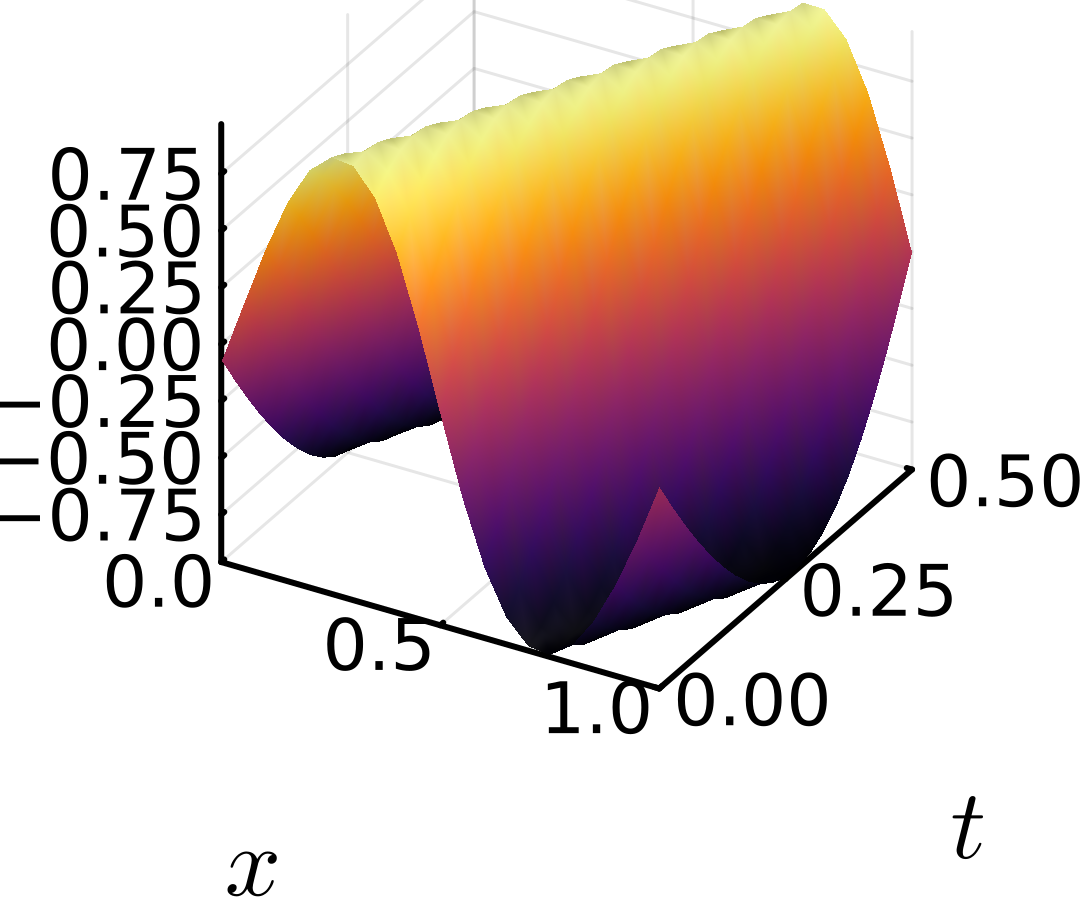}
	\caption{Left: Reference TW. Centre and Right: Identified TW in learned model}\label{fig:TravellingWaveExperiment}
\end{figure}

We have $\max_{i,j} \|\mathrm{DEL}(L_d)_{i,j}(U_{\mathrm{ref}}^{\mathrm{TW}})\| <0.004$, where $U_{\mathrm{ref}}^{\mathrm{TW}}$ is the TW from \cref{ex:twDiscreteWave} ($n=1$). This shows that the exact TW is a solution of the learned discrete field theory. This is remarkable since TWs are not part of the training data. However, \cref{rem:Palais} hints that the ansatz of an autonomous $L_d$ favours TWs as it contains the right symmetries.
Using the method of \cref{sec:TW}, a TW $U^{\mathrm{TW}}$ and speed $c$ can be found numerically: with $(c_1,U_{\mathrm{ref}}^{\mathrm{TW}})$ as an initial guess with normally distributed random noise ($\sigma =0.5$) added to the Fourier coefficients of $U_{\mathrm{ref}}^{\mathrm{TW}}$ and to $c_1$ , we find $U^{\mathrm{TW}}$ and $c$ for the learned $L_d$ with errors $\|U^{\mathrm{TW}}-U_{\mathrm{ref}}^{\mathrm{TW}}\|_\infty < 0.12$ and $|c-c_1| < 0.001$ (using $10^4$ epochs of {\tt adam}).
(\cref{fig:TravellingWaveExperiment})


\section{Conclusion and future work}
We present an approach to learn models of dynamical systems that are governed by an (a priori unknown) variational PDE from data. This is done by learning a neural-network model of a discrete Lagrangian such that the discrete Euler--Lagrange equations (DELs) are fulfilled on the training data. 
As DELs are local, the model can be efficiently trained and used in simulations. Even though the underlying system is infinite dimensional, model order reduction is not required.
It would be interesting to relate the implicit locality assumption of our data-driven model to the more widely used approach to fit a dynamical system on a low-dimensional latent space that is identified using model order reduction techniques \cite{Blanchette2020,Glas2023,Carlsberg2015,Blanchette2022,Sharma2022LagrangianROM}.

Our approach fits in the context of physics-informed machine learning because the data-driven model has (discrete) variational structure by design. However, our model is numerical analysis-informed as well: since our model is discrete by design, it can be used in simulations without an additional discretisation step.
Based on an analysis of Newton's method when used to solve DELs, we develop a regulariser that rewards numerical regularity of the model. The regulariser is employed during the training phase. It plays a crucial role to obtain a non-degenerate discrete Lagrangian.

Our work provides a proof of concept illustrated on the wave equation. It is partly tailored to the hyperbolic character of the underlying PDE.
It is future work to adapt this approach to dynamical systems of fundamentally different character (such as parabolic or elliptic behaviour) by employing discrete Lagrangians and regularisers that are adapted to the information flow within such PDEs. 

Finally, we clarify the notion of travelling waves (TWs) in discrete models and show how to locate TWs in data-driven models numerically. Indeed, in our numerical experiment the data-driven model contains the correct TWs even though the training data does not contain any TWs. In future work it would be interesting to develop techniques to identify more general highly symmetric solutions in data-driven models and use them to evaluate qualitative aspects of learned models of dynamical systems.

\subsubsection{Source Code} 
\url{https://github.com/Christian-Offen/LagrangianDensityML}

\subsubsection{Acknowledgements} 
C. Offen acknowledges the Ministerium für Kultur und Wissenschaft des Landes Nordrhein-Westfalen and computing time provided by the Paderborn Center for Parallel Computing (PC2).

\appendix
\section{Proofs}\label{appendix1}
\begin{proof}[\cref{prop:SolveDELNewton}]
	We adapt the standard estimates for Newton's method (see \cite[\S 4]{Deuflhard2003}, for instance) to the considered setting.
	Let $f \colon \mathcal{O} \to \R^d$ with
	\[
	f(u) = \frac{\p}{\p u^i_j}
	\big(
	L_d(u^{i}_{j},u,u^{i}_{j+1})
	+ L_d(u^{i-1}_{j},u^{i}_{j},u^{i-1}_{j+1})
	+ L_d(u^{i}_{j-1},u^{i+1}_{j-1},u^{i}_{j})
	\big).
	\]
	With this definition, $f(u^\ast)=0$, $\theta$ is a Lipschitz constant for $\D f \colon \mathcal{O} \to \R^{d \times d}$, $u \mapsto \D f(u)$, $\overline{\theta}$ is a Lipschitz constant for $\mathrm{inv} \circ \D f \colon \mathcal{O} \to \R^{d \times d}$, $u \mapsto \D f(u)^{-1}$, and $\rho^\ast = \| \D f(u^\ast)^{-1} \|$.
	Here $\D f(u)$ denotes the Jacobian matrix of $f$ at $u \in \mathcal O$.
	
	Assume that for $n \in \N$ an iterate $u^{(n)} \in \mathcal{O}$ fulfils $\|u^{(n)}-u^\ast\| \le \min\left(\frac{\rho^\ast}{\overline{\theta}}, \frac{1}{2\theta\rho^\ast}\right)$. Then \begin{equation*}\begin{split}
			\| \D f(u^{(n)})^{-1} &\| 
			= \| \D f(u^{(n)})^{-1} - \D f(u^{\ast})^{-1} + \D f(u^{\ast})^{-1}\| \\
			&\le \| \D f(u^{(n)})^{-1} - \D f(u^{\ast})^{-1} \| + \rho^\ast 
			\le \underbrace{\overline{\theta} \|u^{(n)} - u^\ast \|}_{\le \rho^\ast} + \rho^\ast \le 2 \rho^\ast.
	\end{split}\end{equation*}
	The next iterate $u^{(n+1)} = u^{(n)} - \D f(u^{(n)})^{-1} f(u^{(n)})$ can be bounded:
	\begin{align*}
		\| &u^{(n+1)} - u^\ast \| 
		= \| u^{(n)} - u^\ast - \D f(u^{(n)})^{-1} (f(u^{(n)}) - \underbrace{f(u^\ast)}_{=0})  \| \\
		&\le \underbrace{\|\D f(u^{(n)})^{-1}  \|}_{\le 2 \rho^\ast} \Big\| \underbrace{\D f(u^{(n)})(u^{(n)}-u^\ast)}_{=\int_0^1 \D f(u^{(n)})(u^{(n)}-u^\ast) \d t } - \underbrace{(f(u^{(n)}) - f(u^\ast) )}_{\int_0^1 \D f (u^{(n)} + t (u^\ast -u^{(n)}))(u^{(n)} - u^\ast) \d t }  \Big\| \\
		&\le 2 \rho^\ast \left \| \int_0^1 \left(\D f(u^{(n)}) - \D f(u^{(n)} + t(u^\ast - u^{(n)}))\right)(u^{(n)}-u^\ast) \d t \right\|\\
		&\le 2 \rho^\ast  \int_0^1 \underbrace{\left \|\D f(u^{(n)}) - \D f(u^{(n)} + t(u^\ast - u^{(n)}))\right\|}_{\le \theta t \|u^{(n)}-u^\ast\| } \left\|u^{(n)}-u^\ast\right\| \d t\\
		&\le 2 \rho^\ast \theta \left\|u^{(n)}-u^\ast\right\|^2  \int_0^1 t \d t\\
		&=\rho^\ast \theta \left\|u^{(n)}-u^\ast\right\|^2
	\end{align*}
	Moreover, since $\left\|u^{(n)}-u^\ast\right\| \le (2\theta \rho^\ast)^{-1}$, we have
	$\|u^{(n+1)} - u^\ast \| \le  \frac 12 \|u^{(n)} - u^\ast \|$. By induction, the Newton Iterations converge against $u^\ast$ and \eqref{eq:NewtonQuadraticConvergence} holds true.
\end{proof}

%
%
%
\bibliographystyle{splncs04}
\bibliography{resources}

\begin{thebibliography}{10}
\providecommand{\url}[1]{\texttt{#1}}
\providecommand{\urlprefix}{URL }
\providecommand{\doi}[1]{https://doi.org/#1}

\bibitem{Blanchette2020}
Allen-Blanchette, C., Veer, S., Majumdar, A., Leonard, N.E.: {LagNetViP}: A
  {L}agrangian neural network for video prediction ({AAAI} 2020 symposium on
  physics guided ai) (2020). \doi{10.48550/ARXIV.2010.12932}

\bibitem{Bertalan2019}
Bertalan, T., Dietrich, F., Mezi{\'{c} }, I., Kevrekidis, I.G.: On learning
  {H}amiltonian systems from data. Chaos: An Interdisciplinary Journal of
  Nonlinear Science  \textbf{29}(12),  121107 (dec 2019).
  \doi{10.1063/1.5128231}

\bibitem{Glas2023}
Buchfink, P., Glas, S., Haasdonk, B.: Symplectic model reduction of
  {H}amiltonian systems on nonlinear manifolds and approximation with weakly
  symplectic autoencoder. SIAM Journal on Scientific Computing  \textbf{45}(2),
   A289--A311 (2023). \doi{10.1137/21M1466657}

\bibitem{Carlsberg2015}
Carlberg, K., Tuminaro, R., Boggs, P.: Preserving {L}agrangian structure in
  nonlinear model reduction with application to structural dynamics. SIAM
  Journal on Scientific Computing  \textbf{37}(2),  B153--B184 (2015).
  \doi{10.1137/140959602}

\bibitem{LNN}
Cranmer, M., Greydanus, S., Hoyer, S., Battaglia, P., Spergel, D., Ho, S.:
  Lagrangian neural networks (2020). \doi{10.48550/ARXIV.2003.04630}

\bibitem{Deuflhard2003}
Deuflhard, P., Hohmann, A.: Numerical Analysis in Modern Scientific Computing.
  Springer New York (2003). \doi{10.1007/978-0-387-21584-6}

\bibitem{Dierkes2021}
Dierkes, E., Flaßkamp, K.: Learning {H}amiltonian systems considering system
  symmetries in neural networks. IFAC-PapersOnLine  \textbf{54}(19),  210--216
  (2021). \doi{10.1016/j.ifacol.2021.11.080}, 7th IFAC Workshop on Lagrangian
  and Hamiltonian Methods for Nonlinear Control LHMNC 2021

\bibitem{SymHNN}
Dierkes, E., Offen, C., Ober-Blöbaum, S., Flaßkamp, K.: {Hamiltonian neural
  networks with automatic symmetry detection}. Chaos: An Interdisciplinary
  Journal of Nonlinear Science  \textbf{33}(6) (06 2023).
  \doi{10.1063/5.0142969}, 063115

\bibitem{HNN}
Greydanus, S., Dzamba, M., Yosinski, J.: Hamiltonian {N}eural {N}etworks. In:
  Wallach, H., Larochelle, H., Beygelzimer, A., d'Alch\'{e} Buc, F., Fox, E.,
  Garnett, R. (eds.) Advances in Neural Information Processing Systems.
  vol.~32. Curran Associates, Inc. (2019),
  \url{https://proceedings.neurips.cc/paper/2019/file/26cd8ecadce0d4efd6cc8a8725cbd1f8-Paper.pdf}

\bibitem{SymLNN}
Lishkova, Y., Scherer, P., Ridderbusch, S., Jamnik, M., Liò, P.,
  Ober-Blöbaum, S., Offen, C.: Discrete {L}agrangian neural networks with
  automatic symmetry discovery (accepted). In: Proceedings of IFAC World
  Congress 2023, Yokohama, Japan (9-14/07/2023). IFAC-PaperOnLine (2023).
  \doi{10.48550/ARXIV.2211.10830}

\bibitem{MarsdenWestVariationalIntegrators}
Marsden, J.E., West, M.: Discrete mechanics and variational integrators. Acta
  Numerica  \textbf{10},  357--514 (2001). \doi{10.1017/S096249290100006X}

\bibitem{Blanchette2022}
Mason, J., Allen-Blanchette, C., Zolman, N., Davison, E., Leonard, N.: Learning
  interpretable dynamics from images of a freely rotating 3d rigid body (2022).
  \doi{10.48550/ARXIV.2209.11355}

\bibitem{PDEBEA}
McLachlan, R.I., Offen, C.: Backward error analysis for variational
  discretisations of pdes. Journal of Geometric Mechanics  \textbf{14}(3),
  447--471 (2022). \doi{10.3934/jgm.2022014}

\bibitem{LagrangianShadowIntegrators}
Ober-Blöbaum, S., Offen, C.: Variational learning of {E}uler–{L}agrange
  dynamics from data. Journal of Computational and Applied Mathematics
  \textbf{421},  114780 (2023). \doi{10.1016/j.cam.2022.114780}

\bibitem{DLNNDensity}
Offen, C., Ober-Bl{\"o}baum, S.: Learning discrete {L}agrangians
  for variational pdes from data and detection of travelling waves. In:
  Nielsen, F., Barbaresco, F. (eds.) Geometric Science of Information. Lecture
  Notes in Computer Science, vol. 14071, pp. 569--579. Springer Nature
  Switzerland, Cham (2023). \doi{10.1007/978-3-031-38271-0_57}

\bibitem{symplecticShadowIntegrators}
Offen, C., Ober-Blöbaum, S.: Symplectic integration of learned {H}amiltonian
  systems. Chaos: An Interdisciplinary Journal of Nonlinear Science
  \textbf{32}(1),  013122 (1 2022). \doi{10.1063/5.0065913}

\bibitem{palais1979}
Palais, R.S.: The principle of symmetric criticality. Comm. Math. Phys.
  \textbf{69}(1),  19--30 (1979),
  \url{https://projecteuclid.org:443/euclid.cmp/1103905401}

\bibitem{Qin2020}
Qin, H.: Machine learning and serving of discrete field theories. Scientific
  Reports  \textbf{10}(1) (11 2020). \doi{10.1038/s41598-020-76301-0}

\bibitem{Ridderbusch2021}
Ridderbusch, S., Offen, C., Ober-Blobaum, S., Goulart, P.: Learning {ODE}
  models with qualitative structure using {G}aussian {P}rocesses. In: 2021 60th
  {IEEE} Conference on Decision and Control ({CDC}). {IEEE} (12 2021).
  \doi{10.1109/cdc45484.2021.9683426}

\bibitem{Sharma2022LagrangianROM}
Sharma, H., Kramer, B.: Preserving {L}agrangian structure in data-driven
  reduced-order modeling of large-scale dynamical systems (2022).
  \doi{10.48550/ARXIV.2203.06361}

\end{thebibliography}

\end{document}